\documentclass[12pt,reqno]{amsart}
\advance \topmargin by -\headheight
\advance \topmargin by -\headsep
\evensidemargin \oddsidemargin
\usepackage{amsmath}
\usepackage{amsfonts}
\usepackage{stmaryrd}
\usepackage{amssymb}
\usepackage{amsthm}
\usepackage{csquotes}
\usepackage{color}
\usepackage{mathrsfs}
\usepackage{dsfont}
\usepackage{bm}
\usepackage{cite}
\usepackage{soul}

\numberwithin{equation}{section}
\renewcommand\vec{\bm}

\newtheorem{theorem}{Theorem}[section]

\newtheorem{lemma}[theorem]{Lemma}
\newtheorem{Proposition}[theorem]{Proposition}
\newtheorem{Conjecture}[theorem]{Conjecture}

\usepackage{mathtools}

\def\L{\mathcal{L}}
\def\ZZ{\mathbb{Z}}
\def\RR{\mathbb{R}}

\def\FF{\mathbb{F}}
\def\QQ{\mathbb{Q}}
\def\GL{\mathrm{GL}}
\def\M{\mathrm{Mat}}

\title{Brunn--Minkowski type estimates for certain discrete sumsets}

\author{Albert Lopez Bruch}
\address{King's College London, London WC2R 2LS, UK}
\email{albert.lopez\_bruch@kcl.ac.uk}

\author{Yifan Jing}
\address{Department of Mathematics, the Ohio State University, Columbus 43212, USA}
\email{jing.245@osu.edu}

\author{Akshat Mudgal}
\address{Mathematics Institute, Zeeman Building, University of Warwick, Coventry CV4 7AL, UK}
\email{Akshat.Mudgal@warwick.ac.uk}

\thanks{ALB was supported by a Summer Research Internship from Mathematical Institute, Oxford; while YJ and AM were supported by Ben Green’s Simons Investigator Grant, ID:376201. }


\renewcommand\vec{\bm}

\begin{document}

\begin{abstract}
Let $d,k$ be natural numbers and let $\mathcal{L}_1, \dots, \mathcal{L}_k \in \mathrm{GL}_d(\mathbb{Q})$ be linear transformations such that there are no non-trivial subspaces $U, V \subseteq \mathbb{Q}^d$ of the same dimension satisfying $\L_i(U) \subseteq V$ for every $1 \leq i \leq k$. For every non-empty, finite set $A \subset \mathbb{R}^d$, we prove that
\[ |\mathcal{L}_1(A) + \dots + \mathcal{L}_k(A) | \geq k^d |A| - O_{d,k}(|A|^{1- \delta}), \]
where $\delta >0$ is some absolute constant depending on $d,k$. Building on work of Conlon--Lim, we can show stronger lower bounds when $k$ is even and $\mathcal{L}_1, \dots, \mathcal{L}_k$ satisfy some further incongruence conditions, consequently resolving various cases of a conjecture of Bukh. Moreover, given any $d, k\in \mathbb{N}$ and any finite, non-empty set $A \subset \RR^d$ not contained in a translate of some hyperplane, we prove sharp lower bounds for the cardinality of the $k$-fold sumset $kA$ in terms of $d,k$ and $|A|$. This can be seen as a $k$-fold generalisation of Freiman's lemma.
\end{abstract}

\maketitle

\section{Introduction}

A classical result in analysis concerns estimating size of sumsets of arbitrary subsets of Euclidean spaces. In particular, given positive integers $d,k$ and non-empty, compact sets $\mathcal{A}_1, \mathcal{A}_2, \dots, \mathcal{A}_k \subseteq \mathbb{R}^d$, the well-known Brunn--Minkowski inequality states that the sumset
\[ \mathcal{A}_1 + \dots + \mathcal{A}_k = \{a_1 + \dots + a_k : a_1 \in \mathcal{A}_1, \dots, a_k \in \mathcal{A}_k \} \]
satisfies
\begin{equation} \label{bm1}
    \mu(\mathcal{A}_1 + \dots + \mathcal{A}_k) \geq (\mu(\mathcal{A}_1)^{1/d} + \dots + \mu(\mathcal{A}_k)^{1/d} )^d,
\end{equation} 
where we denote $\mu$ to be the Lebesgue measure in $\mathbb{R}^d$. Moreover, such a result can be seen to be sharp, by considering the case when $\mathcal{A}_1, \dots, \mathcal{A}_k$ are homothetic convex sets. It is natural to ask whether such an estimate remains true in the discrete setting. In this endeavour, we define for every finite, non-empty set $A \subset \mathbb{R}^d$, the dimension $\dim(A)$ to be the dimension of the affine span of the set $A$. Analogously to the continuous setting, one may naively expect that whenever $\dim(A) = d$, then the set $kA = \{a_1 + \dots + a_k : a_1, \dots, a_k \in A\}$ satisfies $|kA| \geq (k^d-o(1))|A|$. However, this fails to be true as evinced by the example 
\begin{equation} \label{ex10}
    A_N = \{0,e_2,\dots, e_d\} + \{e_1,2 \cdot e_1,\dots, (N-d) \cdot e_1 \}, 
\end{equation} 
where $\{e_1, \dots , e_d\}$ denotes the canonical basis of $\RR^d$. In particular, $|kA_N|/|A_N| \leq (d+1)^{k-1}$, and so, one can expect the so-called doubling to grow at most polynomially in $d$.

It is natural to analyse the conditions that sets $A,B \subseteq \mathbb{R}^d$ of cardinality $N$ must satisfy so that $|A+B| \geq (C^d -o(1)) N$ holds for some constant $C > 1$. One possible answer to this is given by Green--Tao \cite{GT2006}, who proved that whenever $A \subseteq \mathbb{Z}^d$ contains $\{0,1\}^d$, then $|2A| \geq 2^{d/2} |A|$, see also \cite{MRSZ2022}. A recent breakthrough result of Green--Gowers--Manners--Tao \cite{GGMT2023} on the polynomial Frieman--Ruzsa conjecture implies the existence of constants $C,C'>1$ such that any finite set $A \subset \mathbb{Z}^d$ satisfying $|A+A| < C^d|A|$ must contain $A' \subseteq A$ such that $|A'| \geq |A|/C^{dC'}$ and $\dim(A') \leq d-1$. Both these results have found various applications towards the sum-product conjecture and related topics, see \cite{PZ2020, Mu2023, Mu2023b}.
\par

In this paper, we investigate a third such setting which has been explored in a conjecture of Conlon--Lim \cite{CL2022}, which itself is a revised version of a previous question raised by Bukh \cite[Problem $4$]{Bu2007}. This concerns sumsets of the form 
\[ \L_1(A) + \dots + \L_k(A)  = \{ \L_1(a_1) + \dots + \L_k(a_k) : a_1, \dots, a_k \in A   \}   \]
where $\L_1, \dots, \L_k \in \M_d(\ZZ)$ and $A\subseteq \RR^d$ is a finite set. Following \cite{CL2022}, we denote matrices $\L_1, \dots ,\L_k \in \M_d(\mathbb{Z})$ to be \emph{irreducible} if there are no non-trivial subspaces $U, V \subseteq \mathbb{Q}^d$ of the same dimension such that $\L_i(U) \subseteq V$ for every $1 \leq i \leq k$. Furthermore, we write $\L_1, \dots, \L_k$ to be \emph{coprime} if there do not exist matrices $\mathcal{P},\mathcal{R} \in \GL_d(\QQ)$ satisfying $0 < |\det (\mathcal{P}) \det (\mathcal{R})|< 1$ such that $\mathcal{P} \L_1 \mathcal{R}, \dots, \mathcal{P} \L_k \mathcal{R} \in \M_d(\ZZ)$. With these definitions in hand, we state the following conjecture from \cite{CL2022}.

\begin{Conjecture} \label{cl2}
Given $d,k \in \mathbb{N}$ and irreducible, coprime $\L_1, \dots, \L_k \in \M_d(\ZZ)$ and a finite, non-empty set $A \subseteq \ZZ^d$, one has 
\[ |\L_1(A) + \dots + \L_k(A)| \geq ( \det(\L_1)^{1/d} + \dots + \det(\L_k)^{1/d})^d |A| - o(|A|). \]
\end{Conjecture}

Note that the analogous result holds in the continuous setting by simply applying $\eqref{bm1}$. Moreover, the hypothesis that $\L_1, \dots, \L_k$ be coprime and irreducible is necessary, see \cite[\S1]{CL2022}.

When $d=1$, Conjecture \ref{cl2} matches \cite[Problem $4$]{Bu2007} and
was resolved by Bukh \cite{Bu2008}. On the other hand, when $d \geq 2$, even the simplest non-trivial case of this conjecture remained open until work of the third author \cite[Theorem 1.3]{Mu2019} and subsequent work of Krachun--Petrov \cite[Theorem 1]{KP2020}.
More recently, in their very nice paper \cite{CL2022}, Conlon--Lim completely resolved the $k=2$ case of Conjecture \ref{cl2}. In particular, their main result implies that for any $d \in \mathbb{N}$ and for any coprime, irreducible $\L_1, \L_2 \in \M_d(\ZZ)$, there exists some constant $\delta = \delta(d, \L_1, \L_2) >0$ such that for any finite $A \subset \RR^d$, one has
\begin{equation} \label{jp1}
|\L_1(A) + \L_2(A)| \geq ( \det (\L_1)^{1/d} + \det (\L_2)^{1/d})^d |A| - O_{d, \L_1, \L_2}(|A|^{1 - \delta})  
\end{equation}

Our main aim in this paper is to present a variety of sharp results in the $k \geq 3$ regime of Conjecture $\ref{cl2}$ and related problems. In this endeavour, letting $\mathbb{F}$ be either $\QQ$ or $\RR$, we denote matrices $\L_1, \dots, \L_k \in \GL_d(\mathbb{F})$ to be \emph{$\mathbb{F}$-irreducible} if there are no non-trivial subspaces $U, V \subseteq \mathbb{F}^d$ of the same dimension such that $\L_i(U) \subseteq V$ for every $1 \leq i \leq k$. With this in hand, we present our first result below. 

\begin{theorem} \label{e1}
Let $\mathbb{F}$ be either $\QQ$ or $\RR$. Given $d,k \in \mathbb{N}$, there exists $\delta = \delta(d,k)>0$ such that for any $\mathbb{F}$-irreducible $\L_1, \dots, \L_k$ in  $\GL_d(\mathbb{F})$ and any finite set $A \subseteq \RR^d$, we have
\begin{equation} \label{js3}
    |\L_1(A) + \dots + \L_k(A)| \geq k^d |A| - O_{d,k}(|A|^{1- \delta}).  
\end{equation}
If $\mathbb{F} = \mathbb{Q}$ and $k$ is even and $\L_{2i-1}, \L_{2i} \in \M_d(\ZZ)$ are coprime and irreducible for every $1 \leq i \leq k/2$, then there exists $\delta' = \delta'(\L_1, \dots, \L_k) > 0$ such that for every $A \subseteq \RR^d$, we have
\begin{equation} \label{lux}
    |\L_1(A) + \dots + \L_k(A)|  \geq ( \det(\L_1)^{1/d} + \dots + \det(\L_k)^{1/d})^d |A| - O_{\L_1, \dots, \L_k}(|A|^{1- \delta'}). 
\end{equation} 
\end{theorem}

In particular, we are able to resolve Conjecture \ref{cl2} in the more specific case when $k$ is even and $\L_{2i-1}, \L_{2i}$ are coprime and irreducible for $1 \leq i \leq k/2$, the latter being a more restrictive condition than just $\L_1,\dots, \L_k$ being coprime and irreducible. 

We further note that the main term in \eqref{js3} is sharp for a variety of cases. For instance, we have the following example. 

\begin{Proposition} \label{sharpexample}
    Let $d\geq 2$, let $\L_1$ be the identity matrix, and for each $2\leq j\leq d$, let $\L_j\in \GL_d(\RR)$ be the linear transformation satisfying 
    \[ \L_j(e_1)=e_j \ \ , \ \ \L_j(e_j)=-e_1 \ \ \text{and} \ \ \L_j(e_i)=e_i \ \  \text{for every} \ \ i\in [d] \setminus \{1,j\}.\]
    Then $\L_1,\L_2,\dots,\L_d$ are $\RR$-irreducible and coprime. Moreover, setting 
    \[ A = \{ \vec{x} \in \ZZ^d : |x_1|, \dots, |x_d| \leq N\},\]
    we have that
    \[ |\L_1(A) + \dots + \L_d(A)| = d^d |A| + O_{d}(|A|^{1 - 1/d}). \]
\end{Proposition}

We now turn to another discrete Brunn--Minkowski result proved by  Grynkiewicz--Serra \cite{GS2010} for sumsets in $\RR^2$. Thus, let $A_1, A_2 \subset \RR^2$ be finite, non-empty sets and let $l$ be a line in $\RR^2$. Moreover, let $r_1$ and $r_2$ be the minimal number of translates of $l$ required to cover $A_1$ and $A_2$ respectively. In this setting, Grynkiewicz--Serra \cite{GS2010} proved that
\begin{equation}  \label{gs2010r} 
    |A_1 + A_2| \geq (|A_1|/r_1 + |A_2|/r_2 -1)(r_1 + r_2 - 1),
\end{equation}
consequently generalising earlier work of Freiman \cite{Fr1973} who proved \eqref{gs2010r} when $A_1 = A_2$. As a straightforward consequence of the methods involved in the proof of Theorem \ref{e1}, we can prove a $k$-fold generalisation of the above result.

\begin{theorem} \label{lem:par} 
Given $k \geq 2$, let
$A_1,\ldots,A_k\subseteq\mathbb{R}^2$ be finite, non-empty sets, and let $l$ be a line in $\RR^2$. For each $1\leq i\leq k$, let $r_i$ be the minimal number of translates of $l$ required to cover $A_i$. Then
\begin{equation} \label{qn5}
|A_1+\dots+A_k|\geq\left(\frac{|A_1|}{r_1}+\dots+\frac{|A_k|}{r_k}-(k-1)\right)(r_1+\dots+r_k-(k-1)) . 
\end{equation}
\end{theorem}

We firstly note that \eqref{qn5} is sharp. Indeed, 
let $n_1,\dots,n_k,m_1,\dots,m_k$ be positive integers and let $A_i=\{1,2,\dots,n_i\}\times\{1,2,\dots,m_i\}$ for every $1\leq i\leq k$. Let $l$ be a line parallel to the $x$-axis, so that $r_i=m_i$ and $|A_i|/r_i=n_i$ for every $1\leq i\leq k$. In this case, one sees that 
$$A_1+\dots+A_k=\Big\{k,k+1,\dots,\sum_{i=1}^k n_i\Big\}\times\Big\{k,k+1,\dots,\sum_{i=1}^k m_i\Big\},$$
and therefore,
$$|A_1+\dots+A_k|=\Big(\sum_{i=1}^kn_i-(k-1)\Big)\Big(\sum_{i=1}^km_i-(k-1)\Big).$$
This example shows that equality may hold in \eqref{qn5}, and so, the bound is sharp.

We now return to our earlier example in \eqref{ex10} and note that
\[    |kA_N| 
    = \binom{k+d-1}{d} |A_N| - (k-1)\binom{k+d-1}{d-1},\]
    see for instance Lemma \ref{comp}. This is known to be optimal when $k=2$. Indeed, the well-known Freiman's lemma states that for any finite, non-empty set $A \subseteq \mathbb{R}^d$ with $\dim(A) = d$, one has 
\[ |2A| \geq (d+1)|A| - d(d+1)/2.\]
Combining our methods along with some ideas of Gardner--Gronchi \cite{GG2001}, we can further confirm that the above example is optimal for every $k \geq 3$.

\begin{theorem} \label{manysums}
    Let $d,k$ be positive integers, let $A \subseteq \RR^d$ be a finite, non-empty set with $\dim(A) = d$. Then
\[ |kA| \geq  \binom{k+d-1}{d} |A| - (k-1)\binom{k+d-1}{d-1}.\]
\end{theorem}

Our result provides sharp lower bounds even in the regime when $|A|,d$ are fixed and $k$ becomes large, wherein, the lower bound maybe rewritten as
\[ |kA| \geq {k+d-1 \choose d-1} \Bigg(\frac{k(|A|-d)}{d} + 1\Bigg) . \]
Since $\dim(A) =d$, we have that $|A| \geq d+1$, whence the right hand side above looks like $Q_A(k)$, where $Q_A$ is a polynomial of degree $d$ with rational coefficients that depend on $d$ and $|A|$. This is reminiscent of a classical result of Khovanski\u{\i} \cite{Kh1992} which implies that once $k$ is sufficiently large, then $|kA| = P_A(k)$, where $P_A \in \mathbb{Q}[x]$ is some polynomial of degree at most $d$. These polynomials can be quite difficult to describe explicitly and there has been quite a lot of interest recently in estimating the threshold value of $k$ for which $|kA| = P_A(k)$, see \cite{GSW2023} and the references therein. Our result implies that $P_A(k) \geq Q_A(k)$ for $k$ sufficiently large, and that equality holds whenever $A = A_N$.

Our proofs of Theorems \ref{e1}, \ref{lem:par} and \ref{manysums} build on an assortment of ideas from \cite{CL2022, GT2006}, \cite{GS2010} and \cite{GG2001} respectively. In particular, these involve an amalgamation of additive combinatorial results such as Freiman's theorem and Pl\"{u}nnecke--Ruzsa type sumset inequalities, see \S2 for further details, along with a key ingredient in the form of the the technique of compressions. The latter has been widely employed in the literature to analyse sumsets in higher dimensions, see for instance \cite{BL1989, CL2021, CL2022, GG2001, GT2006, GS2010, Mu2022}. In \S3, we will give a brief outline of this technique and then utilise it to prove a variety of sumset inequalities for $k$-fold sumsets in $\RR^d$. The latter consists of Theorem \ref{lem:par} as well as $k$-fold analogues of the discrete Brunn--Minkowski inequality proven by Conlon--Lim \cite[Lemma 2.1]{CL2022}. We devote \S4 to presenting some further combinatorial and geometric maneuvers required for the proof of Theorem \ref{e1}, the latter being recorded in \S5. We utilise \S6 to record the proof of Proposition \ref{sharpexample}. Finally, we use \S7 to present the proof of Theorem \ref{manysums}, following closely some ideas of Gardner--Gronchi \cite{GG2001}.

\textbf{Notation.} We will use Vinogradov notation, that is, we write $X \gg_{z} Y$, or equivalently $Y \ll_{z} X$, to mean $|Y| \leq C_z X$ where $C_z>0$ is some constant depending on the parameter $z$. For any finite set $X$ and any $k \in \mathbb{N}$, we write $X^k = \{(x_1, \dots, x_k) : x_1, \dots, x_k \in X\}$. Given $d \in \mathbb{N}$ and $x \in \mathbb{R}$, we will use $\vec{v}$ to denote the vector $(v_1, \dots, v_d) \in \RR^d$ and write $x \cdot \vec{v} = (xv_1, \dots, xv_d)$. Given a positive integer $N$, we use $[N]$ to denote the set $\{1,2,\dots, N\}$.

\textbf{Acknowledgements.} We would like to thank Hao Huang and Jeck Lim for helpful comments.


\section{Additive combinatorics preliminaries}

In this section, we gather some preliminary results from additive combinatorics that we will employ
throughout our paper. We begin by noting that for any $d, k \geq 1$ and any finite, non-empty sets $A_1, \dots, A_k \subset \RR^d$, one has
\begin{equation} \label{cdav}
    |A_1 + \dots + A_k| \geq |A_1| + \dots + |A_k| - (k-1).
\end{equation} 
Moreover, this is sharp if one takes $A_1, \dots, A_k$ to be arithmetic progressions with the same common difference.

Our next set of preliminary lemmas concern inverse result on sumsets, that is, these characterise finite sets $A \subset \RR^d$ which satisfy $|A+A| \leq K|A|$ for some fixed constant $K \geq 1$. This is the subject matter of a classic result in additive combinatorics known as Freiman's theorem. In this endeavour, letting $\FF = \QQ$ or $\RR$, we define a \emph{generalised arithmetic progression} $P$ in $\FF^d$ of \emph{additive dimension} $D$ and size $L$ to be a set of the form 
\[ P = \{ \vec{u}_0 + l_1 \cdot \vec{v}_1 + \dots + l_D \cdot \vec{v}_D \  :  \ l_i \in [L_i]\ (1 \leq i \leq D) \},  \]
where $\vec{v}_0,\dots, \vec{v}_D$ are elements of $\FF^d$ and $L_1\dots L_D = L$. We denote $P$ to be \emph{proper} if $|P| = L$. We remark that the additive dimension of $P$ may not be the same as $\dim(P)$.

Note that $P$ has a small sumset, that is, $|P+P| \leq 2^D |P|$. Freiman's theorem \cite{GR2007} roughly states that these are essentially the only such examples.

\begin{lemma} \label{Freiman:thm}
Let $\FF = \QQ$ or $\RR$. For any $K \geq 1$, there exist constants $C_1, C_2 >0$ such that if $A \subset \FF^d$ is a finite, non-empty set with $|A+A|\leq K|A|$, then $A$ is contained in a proper generalised arithmetic progression in $\FF^d$ of arithmetic dimension $s \leq C_1$ and size $L\leq C_2|A|$.
\end{lemma}

We conclude this section by presenting various sumset inequalities which hold over arbitrary abelian groups. We start with the well-known Pl\"{u}nnecke--Ruzsa inequality, see \cite{Pe2012} for a streamlined proof.

\begin{lemma} \label{pr}
    Let $A,B$ be finite, non-empty subsets of some abelian group $G$ such that $|A+B|\leq K|B|$ for some $K \geq 1$. Then for every $m,n \in \mathbb{N} \cup \{0\}$, one has
    \[ |mA - nA| \leq K^{m+n}|B|. \]
\end{lemma}

Finally, we record the additive version of the so-called Ruzsa's triangle inequality. In particular, given any finite, non-empty subsets $U,V,W$ of some abelian group $G$, one has
\begin{equation}\label{sumRTI}
|U||V+W|\leq|V+U||U+W|.
\end{equation}
This can be deduced swiftly from \cite[Proposition 2.1]{Pe2012}.


\section{Compressions}

We will utilise this section to give a brief introduction about the technique of compressions as well as develop the various necessary compression lemmas required in the proofs of Theorems \ref{e1}, \ref{lem:par} and \ref{manysums}. Thus, let $H$ be a hyperplane in $\mathbb{R}^d$ and let $\vec{v}\in\mathbb{R}^d$ be a vector not parallel to $H$. For any $\vec{u} \in H$ and any finite, non-empty set $A \subset \RR^d$, we define
\[ A_{\vec{u}} = \{ \vec{a} \in A \ : \ \vec{a} + x \cdot \vec{v} = \vec{u} \ \ \text{for some} \ \ x \in \mathbb{R}  \}. \]
We now define the \emph{compression $\mathcal{C}_{H,\vec{v}}(A)$} of $A$ onto $H$ with respect to $\vec{v}$ to be the set satisfying
\[  \mathcal{C}_{H,\vec{v}}(A)_{\vec{u}} = \{ \vec{u} + n \cdot \vec{v} \ : \ n \in \{0,1,\dots, |A_{\vec{u}}|-1\}  \}  \]
for every $\vec{u} \in H$.  Whenever the hyperplane $H$ is clear from context, we will write $\mathcal{C}_{\vec{v}}(A)$.
Moreover, a finite set $A\subseteq \RR^d$ is \emph{(H,$\vec{v}$)-compressed} if $\mathcal{C}_{H,\vec{v}}(A) = A$. Finally, we denote a set $A \subseteq \RR^d$ to be a \emph{down set} if $A = \mathcal{C}_{H_i,e_i}(A)$ for every $1\leq i \leq d$, where $\{e_1,\ldots,e_d\}$ denotes the standard basis of $\RR^d$ and $H_i=\{\vec{x}\in\mathbb{R}^d:x_i=0\}$.
\par

It is clear from the definition that $|\mathcal{C}_{H,\vec{v}}(A)|=|A|$ for every pair $(H,\vec{v})$. The key property that we will employ throughout our proofs is that sumsets do not increase in size under compressions.

\begin{lemma}\label{compression}
Let $k\in\mathbb{N}$, and let $A_1,\ldots,A_k \subset \RR^d$ be finite, non-empty sets, let $H\subseteq\mathbb{R}$ be some hyperplane containing $\vec{0}$ and let $\vec{v}\in\mathbb{R}^d$ be a vector not parallel to $H$. Then we have that
\[ \mathcal{C}_{\vec{v}}(A_1 + \dots + A_k) \supseteq \mathcal{C}_{\vec{v}}(A_1) + \dots + \mathcal{C}_{\vec{v}}(A_k).\]
In particular, we get that
\[ |A_1 + \dots + A_k| \geq |\mathcal{C}_{\vec{v}}(A_1) + \dots + \mathcal{C}_{\vec{v}}(A_k) |. \]

\end{lemma}

\begin{proof}
Note that the case $k=1$ is trivial. We will prove this for the case when $k=2$, since the general case follows inductively. Indeed, supposing we have proved Lemma \ref{compression} for every $2 \leq k \leq k_0$, then it follows that
\begin{align*}
    \mathcal{C}_{\vec{v}}(A_1+\ldots+A_{k_0+1})
& \supseteq
\mathcal{C}_{\vec{v}}(A_1 + \dots + A_{k_0}) + \mathcal{C}_{\vec{v}}(A_{k_0 + 1}) \\
& \supseteq
\mathcal{C}_{\vec{v}}(A_1) + \dots + \mathcal{C}_{\vec{v}}(A_{k_0}) + \mathcal{C}_{\vec{v}}(A_{k_0 + 1}),
\end{align*} 
where the above inclusions follow from the $k=2$ and $k = k_0$ case of Lemma \ref{compression} respectively. 

Thus, we focus on the $k=2$ case, and so, let $A,B$ be finite non-empty subsets of $\mathbb{R}^d$. For every $\vec{z} \in H$, note that
\begin{align*}
   C_{\vec{v}}(A+ B)_{\vec{z}}  
   & \supseteq \cup_{\substack{\vec{x}, \vec{y} \in H \\ \vec{x} + \vec{y} = \vec{z} }} \{ \vec{z} + n \cdot \vec{v} : 0 \leq n \leq  |A_{\vec{x}}+ B_{\vec{y}}| -1 \} \\
   & \supseteq \cup_{\substack{\vec{x}, \vec{y} \in H \\ \vec{x} + \vec{y} = \vec{z} }}  \{ \vec{z} + n \cdot \vec{v} : 0 \leq n \leq  |A_{\vec{x}}|+ |B_{\vec{y}}| -2 \} \\
   & =  \cup_{\substack{\vec{x}, \vec{y} \in H \\ \vec{x} + \vec{y} = \vec{z} }}  ( \{ \vec{x} + m_1 \cdot \vec{v} : 0 \leq m_1 \leq  |A_{\vec{x}}|-1 \} + \{ \vec{y} + m_2 \cdot \vec{v} : 0 \leq m_2 \leq  |B_{\vec{y}}|-1 \} ) \\
   & = (C_{\vec{v}}(A) +C_{\vec{v}}(B))_{\vec{z}}, 
\end{align*}
where the second inclusion follows from \eqref{cdav}. Taking a union over all $\vec{z} \in H$ on both sides then delivers the desired result.
\end{proof}

If we drop the condition that $\vec{0}\in H$, then the inequality in the conclusion of Lemma \ref{compression} still  holds since 
\[ \mathcal{C}_{H+\mu \cdot \vec{v},\vec{v}}(A)=\mathcal{C}_{H,\vec{v}}(A)+\mu \cdot \vec{v} \ \ \ \text{holds for any }\mu\in\RR.  \]

As a first application of this technique, we now present the proof of Theorem \ref{lem:par}.

\begin{proof}[Proof of Theorem \ref{lem:par}]
We prove this result by induction on $k$. Our base case is $k=2$ which follows from \eqref{gs2010r}, and so, we may assume that $k \geq 3$ and proceed with the inductive step. 

Let $\{e_1,e_2\}$ be the standard basis of $\mathbb{R}^2$. Since the conclusion is invariant under rotations, we may assume that $l$ is parallel to the $e_2$. We first claim that it suffices to consider the case when $A_1, \dots, A_k$ are down sets. In order to see this, write $\pi:\mathbb{R}^2\to\mathbb{R}$ to be the projection map satisfying $\pi(\vec{x}) = x_1$ for every $\vec{x} \in \RR^2$, and note that $r_i = |\pi(A_i)|$ for every $1 \leq i \leq k$. Setting $H$ to be the subspace spanned by $e_1$, we observe that $\pi(A_i)=\pi(\mathcal{C}_{H_2,e_2}(A_i))$ for every $1 \leq i \leq k$. Applying Lemma \ref{compression}, we get that
\[ |A_1+\dots+A_k|\geq|\mathcal{C}_{H,e_2}(A_1)+\dots+\mathcal{C}_{H,e_2}(A_k)|, \]
whence it suffices to assume that $A_1, \dots, A_k$ are $(H,e_2)$--compressed. Let $H'$ be the subspace spanned by $e_2$. We will now compress $A_1, \dots, A_k$ onto $H'$ with respect to $e_2$. For every $1 \leq i \leq k$, since the set $A_i$ was $(H_2, e_2)$--compressed, we have that $|\pi(A_i)| = |\pi(\mathcal{C}_{H',e_1}(A_i))|$. Combining this observation with Lemma \ref{compression}, we may further assume that the sets $A_1, \dots, A_k$ are also $(H',e_1)$--compressed. In particular, the sets $A_1, \dots, A_k$ are down sets. 

Since $A_{k-1}$ and $A_{k}$ are down sets, we have that $|\pi(A_{k-1} + A_{k})| = r_{k-1} + r_{k} - 1$. We may apply the inductive hypothesis for sets $A_1, \dots, A_{k-2}, A_{k-1} + A_k$ to deduce that
\[ |A_1+\dots+A_k| \geq\left(\frac{|A_1|}{r_1}+\dots+\frac{|A_{k-2}|}{r_{k-2}}+\frac{|A_{k-1}+A_{k}|}{r_{k-1}+r_k-1}-(k-2)\right)(r_1 + \dots + r_k - k+1).  \]
Applying the base case for the sets $A_{k-1}$ and $A_k$ gives us 
\[ \frac{|A_{k-1}+A_{k}|}{r_{k-1}+r_k-1} \geq \frac{|A_{k-1}|}{r_{k-1}}+\frac{|A_k|}{r_k}-1, \]
which may now be substituted into the preceding estimate to deliver the desired bound. This concludes the inductive step, and so, we finish the proof of Theorem \ref{lem:par}.
\end{proof}

For the rest of this section, we will use some specific notation in order to simplify the exposition. In particular, for every $1 \leq i \leq d$, we will use $\mathcal{C}_i$ to denote the compression operator $\mathcal{C}_{H_i, e_i}$, where $\{e_1, \dots, e_d\}$ is the canonical basis of $\mathbb{R}^d$ and $H_i = \{ \vec{x} \in \RR^d : x_i = 0\}.$ Moreover, given a basis $\mathcal{B} =\{\vec{b}_1, \dots, \vec{b}_d\}$ of $\RR^d$ and any $I \subseteq [d]$, we denote $\pi_I:\mathbb{R}^d\to\mathbb{R}^{|I|}$ to be the projection map satisfying 
\begin{equation} \label{zfg}
     \pi_I(\sum_{i=1}^d \lambda_i \cdot \vec{b}_i) = \sum_{i \in I} \lambda_i \cdot \vec{b}_i 
\end{equation}
for every $\lambda_1, \dots, \lambda_d \in \RR$. While we will not explicitly write the particular basis $\mathcal{B}$ in notation of $\pi_I$, it will be clear in all instances which $\mathcal{B}$ we will be using.

With this notation in hand, we state our next lemma.

\begin{lemma} \label{comp2}
Let $\mathcal{B}$ be the standard basis of $\RR^d$. Then for every $I\subseteq[d]$ and $i\in[d]$ and for all finite non-empty sets $A_1,\ldots,A_k\subseteq\RR^d$, one has
\begin{equation} \label{proj89}
|\pi_I(\mathcal{C}_i(A_1)+\dots+\mathcal{C}_i(A_k))|\leq|\pi_I(A_1+\dots+A_k)|
\end{equation}
\end{lemma}

\begin{proof}
The case when $i \notin I$ is relatively straightforward, since in this case, one has
\[ \pi_I(\mathcal{C}_i(A_1)+\dots+\mathcal{C}_i(A_k)) = \pi_I(A_1+\dots+A_k). \]
Thus, we may assume that $i\in I$. Writing $J=I\setminus\{i\}$, we define the projection map $\pi: \RR^{|I|} \to \RR^{|J|}$ as $\pi(\sum_{j \in I} \lambda_j \cdot e_j) = \sum_{j \in J} \lambda_j \cdot e_j$ for every sequence $\{\lambda_{j}\}_{j \in J}$ of real numbers. Note that $\pi_J = \pi \circ \pi_I$. As before, for any $\vec{x} \in \RR^{|J|}$ and any non-empty set $C \subset \RR^{|I|}$, we define $C_{\vec{x}} = \{ \vec{c} \in C: \pi(\vec{c}) = \vec{x}\}$. With this notation in hand, we observe that for every $\vec{x} \in \mathbb{R}^{|J|}$ and every $1 \leq l \leq k$, the set $\pi_I(\mathcal{C}_i(A_l))_{\vec{x}}$ is an arithmetic progression of size at most $|\pi_I(A_l)_{\vec{x}}|$ and common difference $e_i$. Thus for any $\vec{z} \in \mathbb{R}^{|J|}$, one has that
\begin{align*}
  |\pi_I(\mathcal{C}_i(A_1) + \dots + \mathcal{C}_i(A_k))_{\vec{z}}| 
  & = \max_{\substack{\vec{x}_1, \dots, \vec{x}_k \in \RR^{|J|}\\ \vec{x}_1 + \dots + \vec{x}_k = \vec{z} }}   |\pi_I(\mathcal{C}_i(A_1))_{\vec{x}_1}| + \dots + |\pi_I(\mathcal{C}_i(A_k))_{\vec{x}_k}| - (k-1)   \\
  & \leq \max_{\substack{\vec{x}_1, \dots, \vec{x}_k \in \RR^{|J|}\\ \vec{x}_1 + \dots + \vec{x}_k = \vec{z} }}   |\pi_I(A_1)_{\vec{x}_1}| + \dots + |\pi_I(A_k)_{\vec{x}_k}| - (k-1) \\
  & \leq \max_{\substack{\vec{x}_1, \dots, \vec{x}_k \in \RR^{|J|}\\ \vec{x}_1 + \dots + \vec{x}_k = \vec{z} }}   |\pi_I(A_1)_{\vec{x}_1} + \dots + \pi_I(A_k)_{\vec{x}_k}| \\
  & =  |\pi_I(A_1 + \dots + A_k)_{\vec{z}}|,
\end{align*}
where the second inequality follows from \eqref{cdav}. The required estimate \eqref{proj89} may be deduced by summing the above inequality for all $\vec{z} \in \pi_J(A_1 + \dots + A_k)$.
\end{proof}

Our final aim in the section is to prove the following Brunn--Minkowski type inequality for $k$-fold discrete sumsets. The $k=2$ case of this was proven by Conlon--Lim \cite[Lemma 2.1]{CL2022}, with a similar result appearing in earlier work of Green--Tao \cite[Lemma 2.8]{GT2006}.

\begin{lemma} \label{gen:lem:2.1}
For any basis $\mathcal{B}$ of $\mathbb{R}^d$ and any finite, non-empty sets $A_1,\dots,A_k \subset \mathbb{R}^d$, one has
\[|A_1+\dots+A_k|\geq(|A_1|^{1/d}+\dots+|A_k|^{1/d})^d-\sum_{I\subsetneq[d]}(k-1)^{d-|I|}|\pi_I(A_1+\dots+A_k)|.\]
\end{lemma}

\begin{proof} We begin by noting that the above conclusion is invariant under the application of invertible linear maps, whence we may assume that $\mathcal{B}=\{e_1,\ldots,e_k\}$ is the canonical basis of $\RR^d$. Moreover, for each $i \in [d]$, noting Lemmata \ref{compression} and \ref{comp2}, we may further assume that the sets $A_1, \dots, A_k$ are $(H_i,e_i)$--compressed. In particular, this means that the sets $A_1, \dots, A_k$ are down sets.

Given some $\vec{x} \in\mathbb{R}^d$ and a down set $A_i \subset \ZZ^d$, define the sets
\[ \mathcal{Q}_{\vec{x}}=\vec{x}+[-1,0]^d \ \ \text{and} \ \ \mathcal{A}_i = \cup_{\vec{x}\in A_i} \mathcal{Q}_{\vec{x}}. \]
Since $A_i$ is a finite, non-empty down set, we see that $\mathcal{A}_i$ is connected and compact. Writing $\mu$ to be the usual Lebesgue measure in $\mathbb{R}^d$, we note that $\mu(\mathcal{A}_i) = |A_i|$, and so, we may apply the Brunn-Minkowski inequality, see \eqref{bm1},  to deduce that
\begin{equation} \label{space45}
     \mu(\mathcal{A}_1 + \dots + \mathcal{A}_k)  \geq ( |A_1|^{1/d}  + \dots + |A_k|^{1/d} )^d .
\end{equation}
Next, we observe that
\begin{align*} 
\mathcal{A}_1 + \dots + \mathcal{A}_k 
& = A_1 + \dots + A_k  + [-k, 0]^d  \\
& = \cup_{\vec{x} \in A_1 + \dots + A_k + \{-k+1, \dots, -1,0\}^d} \mathcal{Q}_{\vec{x}} \\
& = \cup_{I \subseteq [d]} \cup_{\vec{x} \in S_I} \mathcal{Q}_{\vec{x}}, 
\end{align*}
where for any $I \subseteq [d]$, we define
\[  S_I = \{ \vec{x} \in \ZZ^d  \ \ : \ \  \pi_I(\vec{x}) \in \pi_I(A_1 + \dots + A_k) \ \ \text{and} \ \ -k+1 \leq x_i \leq -1 \ \ \text{if} \ \ i \notin I\}.  \]
Noting the above definition along with the fact that if $\vec{x}$ belongs to a down set in $\ZZ^d$ then $x_j \geq 0$ for all $j \in [d]$, we see that for all distinct $I, I' \subseteq [d]$, one has
\[ |S_I| = (k-1)^{d - |I|} |\pi_I(A_1 + \dots + A_k)| \ \ \text{and} \  \ S_I \cap S_{I'} = \emptyset. \]
Combining this with the preceding discussion, we get that
\begin{align*}
    \mu(  \mathcal{A}_1 + \dots + \mathcal{A}_k ) = \sum_{I \subseteq [d]} |S_I| = \sum_{I \subseteq [d]} (k-1)^{d - |I|} |\pi_I(A_1 + \dots + A_k)|.
\end{align*}
Putting this together with \eqref{space45} delivers the required result.  
\end{proof}


\section{Some further combinatorial and geometric maneuvers}

In this section, we will establish various lemmas that will perform an important role in the proof of Theorem \ref{e1}, with the first two being iterative Pl\"{u}nnecke--Ruzsa type inequalities. We begin by proving the following $k$-fold generalisation of a result of Krachun--Petrov \cite[Lemma 3.1]{KP2020}.

\begin{lemma}
\label{gen:lem:2.3}
Let $k, N \geq 2$ be integers, let $G$ be an abelian group and let $A_1,\dots,A_k\subseteq G$ be sets such that $|A_1|=\dots=|A_k|=N$ and let $X = A_1 + \dots + A_k$. If $|X| \leq KN$ for some $K \geq 1$, then $|X+X| \leq K^{7}N$.
\end{lemma}

\begin{proof} The $k=2$ case of this was proven in \cite{KP2020}, but we provide a proof of this here for the sake of completeness. When $k=2$, we note that the hypothesis combines with Lemma \ref{pr} to deliver the bounds $|3A_1|, |3A_2| \leq K^3 N$. We now apply \eqref{sumRTI} to deduce that
\[ |A_2 + 2A_1| \leq \frac{|3A_1||A_1 +A_2|}{|A_1|} \leq K |3A_1| \leq K^4 N, \]
whereupon, we can apply \eqref{sumRTI} again to see that
\[ |2A_2 + 2A_1| \leq  \frac{|3A_2||A_2 + 2A_1|}{|A_2|} \leq K^3 |A_2 + 2A_1| \leq K^7 N.\]

We now turn to the $k \geq 3$ case. Writing $X' = A_2 + \dots + A_k$, our hypothesis implies that $|X' + A_1| \leq KN$, whence we may apply Lemma \ref{pr} to discern that $|3X'| \leq K^{3}N$. Combining this with \eqref{sumRTI}, we get that
\[ |2X' + A_2 +  A_1| \leq \frac{|2X' + A_2 + A_3| |A_1 + A_3|}{|A_3|} \leq K |3X'| \leq K^4 N, \]
where the second last inequality follows from the fact that $|A_1 + A_i|/|A_1| \leq |X|/N \leq K$ for every $2 \leq i \leq k$. Similarly, we may amalgamate the preceding bound with another application of \eqref{sumRTI} to deduce that
\[   |X+X| = |2X' + 2A_1|  \leq \frac{| 2X' + A_1 + A_2||A_1 + A_2|}{|A_2|}  \leq K | 2X' + A_1 + A_2| \leq K^5 N. \qedhere  \]
\end{proof}

This allows us to prove the following inequality on iterated sums of linear transformations.

\begin{lemma} \label{prlin}
Given $k\geq 2$ and $\L_2, \dots, \L_k \in \GL_d(\RR)$ and some finite, non-empty set $A \subset \RR^d$, writing $X = A + \L_2(A) + \dots + \L_k(A)$, if $|X| \leq K|A|$ for some $K \geq 1$, then 
\[ |X + \L_2(X) + \dots + \L_k(X)| \leq K^{7k+1} |A|. \]
\end{lemma}

\begin{proof}
We will prove this inductively, and so, for any $2 \leq j \leq k$, define
\[ T_j = X + \L_2(X) + \dots + \L_j(X) + \L_{j+1}(A) + \dots + \L_k(A). \]
Moreover, set $T_0 = X$ and $T_1 = X + \L_2(A) + \dots + \L_k(A)$. Our aim is to prove inductively that $|T_j| \leq K^{7j+1}|A|$ for every $0 \leq j \leq k$. Our hypothesis covers the base case $j=0$, and so, we proceed to the inductive step. For any $j \in [k]$, we may apply Lemma \ref{sumRTI} to get that
\[    |T_j| \leq \frac{|\L_j(A) + \L_j(X)||T_{j-1}|}{|\L_j(A)|} \leq   K^{7(j-1) +1}|A + X| \leq K^{7(j-1) + 1} |X + X| \leq K^{7j+1}|A|,   \]
where the second and final inequalities follow from the inductive hypothesis and Lemma \ref{gen:lem:2.3} respectively. This concludes the inductive step, and consequently, we finish the proof of Lemma \ref{prlin}.
\end{proof}

It is worth remarking that the conclusion of Lemma \ref{prlin} does not hold if we set $X=\L_1A+\L_2A+\dots+\L_kA$ for some arbitrary $\L_1\in\GL_d(\QQ)$ and consider upper bounds for $\L_1(X) + \dots + \L_k(X)$. Indeed, let $\L_1,\L_2\in\GL_2(\mathbb{Q})$ and $A\subset\mathbb{Q}^2$ be defined as 
\[\L_1=\begin{pmatrix}
    1 & 1\\
    0 & 1
\end{pmatrix}, \quad \L_2=\begin{pmatrix}
    1 & 1\\
    -1 & 1
\end{pmatrix}\quad\text{and}\quad A=\{(0,i):i\in[N]\}.\]
Then one sees that
\[ \L_1(A)=\L_2(A)=\{(i,i):i\in[N]\}, \]
and so, 
\[ X=\L_1(A)+\L_2(A)=\{(i,i):2\leq i\leq 2N\}.\]
On the other hand, we have
\[ \L_1(X)=\{(2i,i):2\leq i\leq 2N\} \ \ \text{and} \ \ 
 \L_2(X)=\{(2i,0):2\leq i\leq 2N\}. \]
 Hence, we get that
\[|X|=2N-1\leq 2|A|\quad\text{while}\quad |\L_1(X)+\L_2(X)|=|X|^2\geq|A|^2,\]
which would contradict an upper bound of the shape $|\L_1(X) + \L_2(X)| \ll_K |A|$ whenever $|A|$ is sufficiently large. 

\vspace{0.5cm}
We are now ready to prove the following geometric result, which roughly states that for any irreducible $\L_1, \dots, \L_k \in \GL_d(\QQ)$, if $A \subset \ZZ^d$ satisfies $|\L_1(A) + \dots + \L_k(A)| \ll |A|$, then $A$ can not have a significantly large subset lying in some lower dimensional subspace. The $k=2$ case of this was proven by Conlon--Lim \cite[Lemma 2.4]{CL2022}.

\begin{lemma} \label{gen:lem:2.4}
Let $\mathbb{F} = \QQ$ or $\RR$, let $\mathcal{L}_1,\dots,\mathcal{L}_k\in \GL_d(\mathbb{F})$ be $\mathbb{F}$-irreducible and let $A\subset \mathbb{R}^d$ be a finite, non-empty set with $|\mathcal{L}_1(A)+\dots+ \mathcal{L}_k(A)|\leq K|A|$ for some $K\geq 1$. Then for any subspace $U$ of $\mathbb{F}^d$ of dimension $r < d$, one has
\[  \sup_{\vec{x} \in \mathbb{F}^d} |(\vec{x} + U) \cap A| \leq (K|A|)^{1-2^{-r}}. \]

\end{lemma}

\begin{proof}
Our proof proceeds via induction on $r$. We first consider the base cases, that is, when  $r=0$ and $r=1$. The case when $r=0$ is trivial since $|(\vec{x}+U)\cap A|=|\{\vec{x}\}\cap A|\in\{0,1\}$ for any $\vec{x}\in\mathbb{F}^d$. On the other hand, for the $r=1$ case, since $\L_1, \dots, \L_k$ are invertible and $\mathbb{F}$-irreducible, there exist some $i < j$ such that $\L_i(U)$ and $\L_j(U)$ are two distinct $1$-dimensional subspaces of $\mathbb{F}^d$. Writing $A' = A \cap (\vec{x} + U)$, one has that
\[  K|A| \geq |\L_i(A) + \L_j(A)| \geq |\L_i(A') + \L_j(A')|  = |A'|^2,  \]
which concludes the base case analysis. 

We now proceed with the inductive step, and so, let $U$ be some subspace of $\mathbb{F}^d$ of dimension $r < d$. As before, let $\vec{x}$ be some element of $\FF^d$ and let $A' = A \cap (U + \vec{x})$. Since $\L_1, \dots, \L_k$ are invertible and irreducible, there exist some $i < j$ such that $\L_i(U)$ and $L_j(U)$ are two distinct $r$-dimensional subspaces of $\FF^d$. Thus, let  $V = \L_i(U) \cap \L_j(U)$ and let $s \geq 1$ be the smallest integer such that there exist $\vec{a}_1, \dots, \vec{a}_s \in A'$ satisfying $A' \subset \L_i^{-1}(V) + \{\vec{a}_1, \dots, \vec{a}_s\}$. Since $r' = \dim(V) \leq r-1$, we may apply  the inductive hypothesis to get that
\begin{equation} \label{patr}
    |A'| \leq s \sup_{\vec{x} \in \FF^d} |(\vec{x}+ \L_i^{-1}(V)) \cap A'| \leq s \sup_{\vec{x} \in \FF^d} |(\vec{x}+ \L_i^{-1}(V)) \cap A| \leq s (K|A|)^{1 - 2^{-r'}} .
\end{equation} 
Our next claim is that the sets $\L_i(\vec{a}_1) + \L_j(A'),\dots, \L_i(\vec{a}_s) + \L_j(A')$ are pairwise disjoint. Indeed, if, say, $\L_i(\vec{a}_1) - \L_i(\vec{a}_2) = \L_j(\vec{a}') - \L_j(\vec{a})$ for some $\vec{a}, \vec{a}' \in A'$, then we have that 
\[ \vec{a}_1 - \vec{a}_2 \in (A'-A') \cap \L_i^{-1}(\L_j(A'-A')) \subset U \cap \L_i^{-1}(\L_j(U)) = \L_i^{-1}(V) ,\]
but this contradicts the minimality of $s$. We employ the aforementioned claim to discern that
\[ s|A'|  \leq \sum_{t=1}^s |\L_i(\vec{a}_t) + \L_j(A')| \leq |\L_i(A') + \L_j(A')| \leq K|A|. \]
Combining this with \eqref{patr}, we get that
\[ |A'|^2  \leq (K|A|)^{2 - 2^{-r'}} \leq (K|A|)^{2 - 2^{-r+1}},  \]
which concludes our inductive step. This finishes the proof of Lemma \ref{gen:lem:2.4}.
\end{proof}


\section{Proof of Theorem \ref{e1}}

Our main aim of this section is to prove Theorem \ref{e1}, and in this endeavour, we prove the following key proposition.

\begin{Proposition} \label{gen:lem:2.5}
Let $\FF = \QQ$ or $\RR$, let $\mathcal{L}_1,\dots,\mathcal{L}_k\in \GL_d(\mathbb{F})$ be $\FF$-irreducible and let $A\subset \mathbb{F}^d$ be a finite, non-empty set such that $|\mathcal{L}_1(A)+\dots+\mathcal{L}_k (A)|\leq K|A|$ for some $K\geq 1$. Then there exists some basis $\mathcal{B} \subset \FF^d$ such that for every $j \in [d]$, one has
\begin{equation} \label{amcp}
    |\pi_{[d]\setminus {j}}(\L_1(A) + \dots + \L_k(A))| \ll_{d,k,K} |A|^{1-\sigma},
\end{equation}
 for some $\sigma \gg_{d,k,K} 1$, where the map $\pi_{[d]\setminus {j}}$ is defined as in \eqref{zfg}.
\end{Proposition}

\begin{proof}
It is worth noting that if $\L_1, \dots, \L_k \in \GL_d(\FF)$ are $\FF$-irreducible, then so are the matrices $\mathcal{I}_d, \L_1^{-1}\circ \L_2, \dots, \L_1^{-1} \circ \L_k$, where $\mathcal{I}_d$ is the $d \times d$ identity matrix. Indeed, if there were some subspaces $U, V$ of $\FF^d$ such that $U \subseteq V$ and $\L_1^{-1}( \L_j(U)) \subseteq V$ for every $2 \leq j \leq k$, then $\L_j(U) \subseteq \L_1(V)$ for every $1 \leq j \leq k$. Moreover, the hypothesis implies that the set 
\[ Y = A + \L_1^{-1}(\L_2(A)) + \dots + \L_1^{-1}(\L_k(A))\]
satisfies $|Y| \leq K|A|$. 

Our aim will be to show that there exists some basis $\mathcal{B}' = \{\vec{b}_1', \dots, \vec{b}_d'\}$ of $\FF^d$ such that for every $j \in [d]$, the map $f_j: \FF^d \to \FF^{d-1}$ satisfying 
\[ f_j( \sum_{i=1}^d \lambda_i \cdot \vec{b}_i') = \sum_{i \in [d]\setminus \{j\} } \lambda_i \cdot \vec{b}_i'  \ \ \text{for every} \ \ \lambda_{1}, \dots, \lambda_d \in \FF \]
further satisfies 
\begin{equation} \label{shw2}
    |f_j(Y)| \ll_{d,k,K} |A|^{1 - \sigma} 
\end{equation}
for some $\sigma \gg_{d,k,K} 1$. This will suffice to prove the desired result, indeed, define the set $X = \L_1(A) + \dots + \L_k(A)$ and consider the basis $\mathcal{B} = \{ \L_1(\vec{b}_1'), \dots, \L_1(\vec{b}_s')\}$ of $\FF^d$. Then for any $I \subseteq [d]$, one has
\[  \sum_{i \in I} \lambda_i \cdot \vec{b}_i' \in Y \ \ \text{if and only if} \ \  \sum_{i \in I} \lambda_i \cdot \L_1(\vec{b}_i') \in X.  \]
Moreover, since the sets $\mathcal{B}', \mathcal{B}$ are bases of $\FF^d$, we may use \eqref{shw2} to deduce that for every $j \in [d]$, one has 
\[ |\pi_{[d]\setminus\{j\}}(X)| = |f_j(Y)| \ll_{d,k,K} |A|^{1 - \sigma}, \]
for some $\sigma \gg_{d,k,K} 1.$

Hence, we will now proceed to construct some $\mathcal{B}'$ such that one has \eqref{shw2}. Note that $|Y| \leq K|A|$, and so, we may apply Lemma \ref{gen:lem:2.3} to deduce that $|Y+Y| \leq K^7|A|$. We may now apply Lemma \ref{Freiman:thm} to discern that $Y$ is contained in some generalised arithmetic progression $P$ where 
\[ P= \{ \vec{v}_0 + l_1 \cdot \vec{v}_1 + \dots + l_s \cdot \vec{v}_s \ : \ l_i \in [L_i] \ \ ( 1 \leq i \leq s)  \}  \]
for some $\vec{v}_0, \dots, \vec{v}_s \in \FF^d$ and  $s \ll_K 1$ and  $L_1, \dots, L_s \in \mathbb{N}$ satisfying $L_1 \geq \dots \geq L_s$ and $L_1 \dots L_s \ll_K |A|$. Since $|A|$ is sufficiently large in terms of $d,k,K$, we may apply Lemma \ref{gen:lem:2.4} to deduce that $A$ can not be completely contained in a lower dimensional subspace of $\FF^d$. Moreover, since $P$ contains a translate of $A$, we get that $\mathrm{Span}_{\mathbb{F}}(\{\vec{v}_1,\dots, \vec{v}_s\})=\FF^d$, where for any finite, non-empty set $T \subset \FF^d$ we define 
\[ \mathrm{Span}_{\mathbb{F}}(T) = \{ \sum_{\vec{t}\in T} \lambda_t \cdot \vec{t} : \lambda_{\vec{t}} \in \FF \}.\]
We now construct a basis $\mathcal{B}' = \{\vec{v}_{i_1}, \dots, \vec{v}_{i_d}\} \subseteq \{\vec{v}_1,\dots, \vec{v}_s\}$ greedily by letting $\vec{v}_{i_1}=\vec{v}_1$, and for $2\leq j\leq d$, letting
\[ i_j=\min\{l: \vec{v}_l\not\in\mathrm{Span}_{\QQ}(\{\vec{v}_{i_1},\dots,\vec{v}_{i_{j-1}}\})\}. \]

As $Y \subseteq P$, we get that
\[ |f_j(Y)| \leq |f_j(P)| \leq \frac{L_1\dots L_s}{L_{i_j}} \ll_K \frac{|A|}{L_{i_j}} .   \]
Since $L_1 \geq \dots \geq L_s$, our goal of proving \eqref{shw2} reduces to showing that $L_{i_d} \gg |A|^{\sigma}$ for some $\sigma \gg_{d,k,K} 1$. Thus, let $H=\mathrm{Span}_{\FF}(\mathcal{B}\setminus\{\vec{v}_{i_d}\})$. Since $Y =A + \L_1^{-1}(\L_2(A)) + \dots + \L_1^{-1}(\L_k(A))$ and $|Y|  \leq K|A|$, we may apply Lemma \ref{prlin} to deduce that 
\[ |Y + \L_1^{-1}(\L_2(Y)) + \dots + \L_1^{-1}(\L_k(Y))| \leq K^{7k+1}|Y|.\]
Moreover, as $\dim(H) = d-1$ and $\mathcal{I}_d, \L_1^{-1}\circ\L_2, \dots, \L_1^{-1}\circ\L_k$ are $\FF$-irreducible, we can apply Lemma \ref{gen:lem:2.4} to get that
\[ \sup_{\vec{x} \in \FF^d}|( H + \vec{x}) \cap Y| \ll_{d,k, K} |A|^{1 - 2^{-d+1}} L_{i_d}^{s - d + 1},. \]
Note that we can cover $P$, and hence $Y$, by at most $L_{i_d}\dots L_s$ many translates of $H$, and so, 
\[ |A| \leq |Y| \leq L_{i_d}\dots L_s \sup_{\vec{x} \in \FF^d}|( H + \vec{x}) \cap Y| \ll_{d,k,K}  |A|^{1 - 2^{-d+1}} L_{i_d}^{s}. \]
This simplifies to give $L_{i_d} \gg_{d,k,K} |A|^{\sigma}$ for some $\sigma \gg_{d,k,K} 1$, and so, we are done.
\end{proof}

We are now ready to present the proof of Theorem \ref{e1}.

\begin{proof}[Proof of Theorem \ref{e1}]
We begin by deducing \eqref{js3} from a combination of Lemma \ref{gen:lem:2.1} and Proposition \ref{gen:lem:2.5}. We begin by claiming that we may assume that $A \subset \FF^d$, which is trivial when $\FF = \RR$ and which may be deduced from the discussion following \cite[Theorem 6.3]{CL2022} when $\FF = \QQ$. We may further assume that $|\mathcal{L}_1(A)+\cdots +\mathcal{L}_k(A)| < 2k^d|A|$, since otherwise we are done. We can now apply Proposition \ref{gen:lem:2.5} with $K=2k^d$ to get find some basis $\mathcal{B}$ of $\FF^d$ satisfying $\eqref{amcp}$. We now combine this with the conclusion of Lemma \ref{gen:lem:2.1} to deduce that
\begin{align*}
     |\mathcal{L}_1(A)+\dots+\mathcal{L}_k(A)|
     & \geq k^d |A| - O_{d,k}(\sum_{j=1}^d |\pi_{[d]\setminus\{j\}}(\L_1(A) + \dots + \L_k(A) )|)  \\
     & \geq k^d |A| - O_{d,k}(|A|^{1 - \sigma})
\end{align*} 
for some $\sigma \gg_{d,k} 1$, which is precisely the content of \eqref{js3}.

We now proceed to prove \eqref{lux}. As in the above proof, noting the discussion following \cite[Theorem 6.3]{CL2022}, it suffices to prove Theorem \ref{e1} for finite, non-empty sets $A \subset \mathbb{Z}^d$. Next, we can translate $A$ so as to ensure that $0\in A$. Writing $X = \L_1(A) + \dots + \L_k(A)$ and $\Lambda = (\det(\L_1)^{1/d} + \dots + \det(\L_k)^{1/d})^d$, we may assume that $|X| \leq 2\Lambda |A|$, since otherwise we are done. Applying Proposition \ref{gen:lem:2.5}, we find a basis $\mathcal{B}$ of $\QQ^d$ satisfying $\eqref{amcp}$. We now apply Lemma \ref{gen:lem:2.1} with $A_j = \L_{2j-1}(A) + \L_{2j}(A)$ for every $1 \leq j \leq k/2$, whence we find that
\begin{align} \label{weather}
     |X|
     & \geq (\sum_{j=1}^{k/2} |\L_{2j-1}(A) + \L_{2j}(A)|^{1/d})^d - O_{d,k}(\sum_{j=1}^d |\pi_{[d]\setminus\{j\}} (\L_1(A) + \dots + \L_k(A))|) \nonumber  \\
     & \geq (\sum_{j=1}^{k/2} |\L_{2j-1}(A) + \L_{2j}(A)|^{1/d})^d - O_{d,k,\Lambda}(|A|^{1 - \sigma}),
\end{align}
for some $\sigma \gg_{d,k, \Lambda} 1$. Noting the hypothesis that $\L_{2j-1}, \L_{2j}$ are coprime and irreducible for every $1 \leq j \leq k/2$, we may now apply the main result of Conlon--Lim \cite[Theorem 1.5]{CL2022}, see also \eqref{jp1}, to deduce that 
\[ |\L_{2j-1}(A) + \L_{2j}(A)| \geq (\det (\L_{2j-1})^{1/d} + \det(\L_{2j})^{1/d} )^d |A| - O_{d,k ,\Lambda}(|A|^{1 - \sigma_j}) \]
for every $1 \leq j \leq k/2$, where $\sigma_1, \dots, \sigma_{k/2} \gg_{d,k, \Lambda} 1$. Upon performing some elementary computations, one may further deduce that
\[  |\L_{2j-1}(A) + \L_{2j}(A)|^{1/d} \geq  (\det (\L_{2j-1})^{1/d} + \det(\L_{2j})^{1/d} ) |A|^{1/d} - O_{d,k, \Lambda}(|A|^{1/d - \sigma_j}). \]
Putting this together with \eqref{weather} and writing $\sigma' = \min_{1 \leq i \leq k} \sigma_i$, we get that
\begin{align*}
    |X| 
    & \geq (1 - O_{d,k,\Lambda}(|A|^{-\sigma'}) )^d  \Lambda |A| - O_{d,k \Lambda}(|A|^{1 - \sigma}) \\
    & \geq \Lambda |A| - O_{d,k, \Lambda }(|A|^{ 1- \delta}),
\end{align*}
where $\delta = \min\{\sigma', \sigma\}$ satisfies $\delta \gg_{d,k, \Lambda} 1$. This concludes the proof of Theorem \ref{e1}.
\end{proof}


\section{Proof of Proposition \ref{sharpexample}}

In this section, we will record the proof of Proposition \ref{sharpexample}. We can immediately deduce that $\L_1,\dots,\L_d\in\M_d(\ZZ)$ are coprime since $\det(\L_i)=1$ for every $1\leq i\leq d$. We will now proceed to show that $\L_1, \dots, \L_d$ are $\RR$-irreducible. The case when $d=2$ is trivial, since in this case, $\L_2$ rotates vectors anticlockwise by $\pi/2$, and so, we may assume that $d\geq 3$. Our proof of this case follows via contradiction, and so, suppose that $\L_1,\L_2,\dots,\L_d$ are not $\RR$-irreducible. Thus, there exist non-trivial subspaces $U,V$ of $\RR^d$ of the same dimension such that $\L_j(U)=V$ for all $j\in [d]$. Since $\L_1$ is the identity matrix, we must have $U = V$, that is, $V$ is an invariant subspace of $\L_j$ for every $j \in [d]$.

Now, choose some $\vec{a} \in V \setminus \{0\}$ and let $S_{\vec{a}} \subset\RR^d$ be the orbit of $\vec{a}$ under the action of $\L_j$ for every $j \in [d]$. Since $V$ is invariant under $\L_1, \dots, \L_d$, it follows that $S_{\vec{a}} \subset V$. Moreover, since $\dim V<d$ by assumption, there must exist some hyperplane $H$ defined by the equation $c_1 x_1 + \dots + c_d x_d = 0$ such that $S_{\vec{a}} \subset V \subset H$. We will now show that such a hyperplane can not exist.

For any $j,k,l \in [d]$ such that $1 < j$ and $k < l$, we define  $\vec{v}_{k,l} = \L_l(\L_l(\L_k(\L_k(\vec{a}))))$ and $\vec{v}_j = \L_j(\L_j(\vec{a}))$. Note that both
\[ \vec{v}_j = (-a_1,  \dots, -a_j,  \dots, a_d) \ \ \text{and} \ \  \vec{v}_{k,l} = (a_1, \dots, -a_k, \dots, -a_l, \dots, a_d)  \]
are elements of $S_{\vec{a}}$. Given $1 \leq k< l \leq d$,  since $\vec{a}, \vec{v}_{k,l} \in S_{\vec{a}} \subset H$, we see that
\[ c_k a_k + c_l a_l = 0 . \]
Thus, for any distinct $l,l' \in [d]$, choosing some $k \in [d]\setminus\{l,l'\}$ we may apply the preceding equality twice to deduce that
\[ c_l a_l = - c_k a_k = c_{l'} a_{l'} . \]
Combining this with the fact that $c_{l'}a_{l'} + c_l a_l = 0$, we discern that 
\[ c_l a_l = 0 \ \ \text{for every} \ \ l \in [d].\] 
Now, for any $1 \leq j \neq  k  \leq d$, we see that the $k^{\text{th}}$ coordinate of $\L_k(\L_j(\vec{a}))$ is either $a_j$ or $-a_j$. Thus, repeating the above procedure with $\L_k(\L_j(\vec{a}))$ instead of $\vec{a}$, we may deduce that \[ c_k a_j = 0 \ \ \text{for any} \ \ 1 \leq j \neq  k  \leq d. \]
Moreover, since there must exist some $k \in [d]$ such that $c_k \neq 0$, the above implies that $a_j = 0$ for every $j \in [d]$. This contradicts the fact that $\vec{a} \in V \setminus \{0\}$, and so, we finish the proof of our claim. 

Finally, let $A = \{ \vec{x} \in \ZZ^d : |x_1|, \dots, |x_d| \leq N\}$. Note that $\L_1(A) = \dots = \L_d(A) = A$. Thus, we have that
\[ |\L_1(A) + \dots + \L_d(A)| = |dA| = (2dN + 1)^d = d^d |A| + O_d(|A|^{1 - 1/d}), \]
which is the desired estimate.


\section{Proof of Theorem \ref{manysums}}

We begin this section by recording \cite[Lemma 3.7]{GG2001} which states that for any finite, non-empty set $A \subset \mathbb{Z}^d$ with $|A| = N$ and $\dim(A) = d$, there exists a finite sequence of compressions which transform $A$ into a long simplex of the form
\[ A_{d,N}  = \{ 0,e_2, \dots, e_d\} \cup \{e_1, 2 \cdot e_1, \dots, (N - d) \cdot e_1\}. \]

\begin{lemma} \label{gardnergronchi}
For any finite, non-empty set $A \subset \mathbb{Z}^d$ with $|A| = N$ and $\dim(A) = d$, there exists some $l \in \mathbb{N}$ and vectors $\vec{v}_1, \dots, \vec{v}_l \in \ZZ^d$ and hyperplanes $H_1, \dots, H_l \in \mathbb{R}^d$, where for every $1 \leq i \leq l$ we have $H_i = \{\vec{x}\in \mathbb{R}^d: x_j = 0\}$ for some $1\leq j \leq d$, such that 
\[ \mathcal{C}_{H_l, \vec{v}_l}( \dots (\mathcal{C}_{H_2, \vec{v}_2}(\mathcal{C}_{H_1, \vec{v}_1}(A))) \dots ) = A_{d,N}. \]
\end{lemma}

The second ingredient we will require for our proof of Theorem \ref{manysums} concerns estimating the size of iterated sumsets of these long simplices. 

\begin{lemma} \label{comp}
    Let $k,N,d$ be positive integers such that $N \geq d+1$. Then 
    \[ |kA_{d,N}| = \binom{k+d-1}{d}(N-d)  + \binom{k+d -1}{d-1} \]
\end{lemma}

\begin{proof}
We will prove this via induction on $d$. The $d=1$ case of this is trivial. Indeed, note that $kA_{1,N} = \{0,e_1,\dots, kN-k\}$, which gives us $|kA_{1,N}| = k(N-1) +1.$ We now move to the inductive step, and so, suppose that we have proved Lemma \ref{comp} for all $d \leq D-1$, for some integer $D \geq 2$. Now observe that
    \[ kA_{D,N} = \cup_{i=0}^k ( (k-i) \cdot e_D + i A_{D-1,N-1}) ,\]
    whence, 
    \[ |kA_{D,N}|  = 1 + \sum_{i=1}^k  |i A_{D-1,N-1}| \]
Applying the inductive hypothesis for $|i A_{D-1,N-1}|$ for every $1 \leq i \leq k$, we get that
\begin{align*}
    |kA_{D,N}| &  
    = 1 + \sum_{i=1}^k \Bigg( \binom{i+D-2}{D-1} (N-D) +  \binom{i+D-2}{D-2} \Bigg) \\
    & = \binom{k-1 + D}{D} (N-D) + \binom{k-1 + D}{D-1} ,
\end{align*}
where the last step follows from applying the standard binominal identity
 \[  \binom{n+l}{l}  = 1 + \binom{l}{l-1}+ \dots +\binom{n+l-2}{l-1} + \binom{n+l-1}{l-1}   \]
twice, once with $(n,l) = (k-1, D)$ and the other time with $(n,l) = (k, D-1)$. This concludes the inductive step, and so, we have proved Lemma \ref{comp}.
\end{proof}

We are now ready to prove Theorem \ref{manysums}.

\begin{proof}[Proof of Theorem \ref{manysums}]
Let $A \subset \RR^d$ be a set such that $\dim(A) = d$ and $|A| = N$ for some $N \geq d+1$. We begin by claiming that it suffices to consider the case when $A \subset \mathbb{Z}^d$ such that $\{0,e_1, \dots, e_d\} \subseteq A$. In order to see this, first translate $A$ so as to ensure that $0 \in A$. We now consider the largest subset $S = \{\vec{s}_1, \dots, \vec{s}_r\} \subseteq A$ which is linearly independent over $\QQ$. Note that $r \geq d$ since $\dim(A) = d$. Moreover, for any $\vec{a} \in A$, there exist $q_1, \dots, q_r \in \mathbb{Q}$ such that $\vec{a} = \sum_{i=1}^r q_i \cdot \vec{s}_i$. We may now consider the map $\psi: A \to \mathbb{Z}^{r}$ which satisfies $\psi(\sum_{i=1}^r q_i \cdot \vec{s}_i) =  (Mq_1, \dots, Mq_r)$, where $M$ is a suitably chosen positive integer which ensures that $\psi(\vec{a}) \in \mathbb{Z}^r$ for every $\vec{a} \in A$. Next, since $\{0,e_1, \dots, e_r\} \subseteq \psi(A) \subset \mathbb{Z}^r$ is a finite set and $r \geq d$, we may find some suitably large integer $X>0$ such that the map $\sigma : \mathbb{Z}^r \to \mathbb{Z}^d$ satisfying
\[ \sigma((b_1, \dots, b_r)) =  (b_1, \dots, b_{d-1}, b_d + X b_{d+1} + \dots + X^{r-d} b_r) \]
for every $(b_1, \dots, b_r) \in \psi(A)$ further satisfies the fact that for every $\vec{a}_1, \dots, \vec{a}_{2k} \in A$, one has
\[  \sum_{i=1}^k (\psi(\vec{a}_i) - \psi(\vec{a}_{k+i})) = 0 \ \ \text{if and only if} \ \ \sum_{i=1}^k (\sigma(\psi(\vec{a}_i)) - \sigma(\psi(\vec{a}_{k+i})) = 0 .\]
 In particular, writing $A' = \sigma(\psi(A))$, we see that $\{0,e_1, \dots, e_d\} \subseteq A' \subset \mathbb{Z}^d$, and $|A'| = N$, and $|kA'|  = |kA|$, whence our aforementioned claim is justified.

Noting Lemma \ref{gardnergronchi}, we may apply Lemma \ref{compression} iteratively to deduce that 
 \[ |kA| \geq |k\mathcal{C}_{H_1, \vec{v}_1}(A)| \geq \dots \geq  |k\mathcal{C}_{H_l, \vec{v}_l}( \dots (\mathcal{C}_{H_2, \vec{v}_2}(\mathcal{C}_{H_1, \vec{v}_1}(A))) \dots )| = |kA_{d,N}|. \]
We finish our proof by applying Lemma \ref{comp} which implies that
\begin{align*}
    |kA|  \geq |kA_{d,N}|  & = \binom{k+d-1}{d}(N-d)  + \binom{k+d -1}{d-1} \\
    & = \binom{k+d-1}{d}|A| - (k-1) \binom{k+d -1}{d-1}. \qedhere
\end{align*} 
\end{proof}


\bibliographystyle{amsbracket}

\begin{thebibliography}{18}


\bibitem{BL1989}
B. Bollob\'{a}s, I. Leader, \emph{Compressions and isoperimetric inequalities}, J. Combin. Theory Ser. A \textbf{56} (1991), no.1, 47--62.

\bibitem{Bu2007}
B. Bukh, \emph{Problems}, available online at http://www.borisbukh.org/problems.html .

\bibitem{Bu2008}
B. Bukh, \emph{Sums of dilates}, Combin. Probab. Comput. \textbf{17} (2008), no. 5, 627--639.


\bibitem{CL2021}
D. Conlon, J. Lim, \emph{Difference sets in $\mathbb{R}^d$}, to appear in Israel J. Math, arXiv:2110.09053.


\bibitem{CL2022}
D. Conlon, J. Lim, \emph{Sums of linear transformations}, to appear in Trans. Amer. Math. Soc., arXiv:2203.09827.


\bibitem{Fr1973}
G. Freiman, \emph{Foundations of Structural Theory of Set Addition}, Transl. Math. Monographs \textbf{37}, Amer. Math. Soc. Providence, R.I. 1973.

\bibitem{GG2001}
R. J. Gardner, P. Gronchi, \emph{A Brunn-Minkowski inequality for the integer lattice}, Trans. Amer. Math.
Soc. \textbf{353} (2001), no. 10, 3995-4024.


\bibitem{GGMT2023}
W. T. Gowers, B. Green, F. Manners, T. Tao, \emph{On a conjecture of Marton}, arXiv:2311.05762.



\bibitem{GSW2023}
A. Granville, G. Shakan, A. Walker, \emph{Effective results on the size and structure of sumsets}, Combinatorica \textbf{43} (2023), no.6, 1139-1178.


\bibitem{GR2007}
B. Green, I. Z. Ruzsa, \emph{Freiman’s theorem in an arbitrary abelian group}, J. Lond. Math. Soc. (2) \textbf{75} (2007), no. 1, 163-175.


\bibitem{GT2006}
B. Green, T. Tao, \emph{Compressions, convex geometry and the Freiman-Bilu theorem}, Q. J. Math. \textbf{57} (2006), no. 4, 495-504.

\bibitem{GS2010}
D. Grynkiewicz, O. Serra, \emph{Properties of two-dimensional sets with small sumset}, J. Combin. Theory Ser. A \textbf{117} (2010), no. 2, 164-188.


\bibitem{Kh1992}
A. G. Khovanski\u{\i}, \emph{The Newton polytope, the Hilbert polynomial and sums of finite sets}, 
Funct. Anal. Appl. \textbf{26} (1992), no. 4, 276-281 (1993).

\bibitem{KP2020}
D. Krachun, F. Petrov, \emph{On the size of $A+ \lambda A$ for algebraic $\lambda$}, Mosc. J. Comb. Number Theory \textbf{12} (2023), no.2, 117-126.

\bibitem{MRSZ2022}
D. Matolcsi, I. Z. Ruzsa, G. Shakan, D. Zhelezov, \emph{An Analytic Approach to Cardinalities of Sumsets,} Combinatorica \textbf{42}, 203-236 (2022).

\bibitem{Mu2019}
A. Mudgal, \emph{Sums of linear transformations in higher dimensions}, Q. J. Math. \textbf{70} (2019), 965-984.

\bibitem{Mu2022}
A. Mudgal, \emph{New lower bounds for cardinalities of higher dimensional difference sets and sumsets}, Discrete Analysis 2022, Paper No. 15, 19 pp.

\bibitem{Mu2023}
A. Mudgal, \emph{Unbounded expansion of polynomials and products}, to appear in Mathematische Annalen, arXiv:2303.15910.

\bibitem{Mu2023b}
A. Mudgal, \emph{ An Elekes--R\'{o}nyai theorem for sets with few products},  Int. Math. Res. Not. IMRN (2024), no. 13, 10410--10424.

\bibitem{PZ2020}
 D. P\'{a}lv\"{o}lgyi, D. Zhelezov, \emph{Query complexity and the polynomial Freiman-Ruzsa conjecture}, Adv. Math. \textbf{392} (2021), Paper No. 108043, 18 pp.

\bibitem{Pe2012}
 G. Petridis, \emph{New proofs of Pl\"{u}nnecke-type estimates for product sets in groups}, Combinatorica \textbf{32}
(2012), no. 6, 721--733.



\end{thebibliography}
\providecommand{\bysame}{\leavevmode\hbox to3em{\hrulefill}\thinspace}

\end{document}